\documentclass[11pt]{amsart}
\usepackage{amssymb,verbatim,graphicx,fullpage}
\usepackage[all]{xy}

\theoremstyle{plain}
\newtheorem{theorem}{Theorem}

\newtheorem{proposition}{Proposition}

\theoremstyle{definition}

\newcommand{\RR}{\mathbb{R}}
\newcommand{\ZZ}{\mathbb{Z}}

\newcommand{\PLFp}{\mathsf{PLF_+}}

\begin{document}

\title[Orders on free groups]{Explicit left orders on free groups extending the lexicographic order on free monoids}

\author{Zoran \v{S}uni\'c}
\address{Department of Mathematics, Texas A\&M University, MS-3368, College Station, TX 77843-3368, USA}
\email{sunic@math.tamu.edu}

\begin{abstract}
For every finitely generated free group we construct an explicit left order extending the lexicographic order on the free monoid generated by the positive letters. The order is defined by a left, free action on the orbit of 0 of a free group of piecewise linear homeomorphisms of the line. The membership in the positive cone is decidable in linear time in the length of the input word. The positive cone forms a context-free language closed under word reversal.
\end{abstract}

\maketitle


A group $G$ is left orderable if there exists a linear order $\leq$ on $G$ that is compatible with the left multiplication, i.e., for all elements $f$, $g$ and $h$ in $G$, if $f \leq g$, then $hf \leq hg$. It has been known at least since the 1940's that the free group $F_k$ of rank $k$ is left orderable (and, in fact, bi-orderable, admitting order that is compatible with both the left and the right multiplication simultaneously). In two of his papers~\cite{neumann:ordered,neumann:ordered2} related to the subject Neumann mentions that, in addition to himself, several other authors have stated this fact in published or unpublished works, including Tarski, G.~Birkhoff, Shimbireva, and Iwasawa (despite his laudable effort to give credit to all, he was unaware of the simultaneous work of Vinogradov~\cite{vinogradov:ordered}). However, most of the early proofs are nonconstructive or too involved (often because of an attempt for greater generality).

Perhaps the most explicit, currently known, construction of an order on free groups is given by the Magnus-Bergman approach~\cite{bergman:ordered}, based on the Magnus embedding~\cite{magnus:free-embedding} of the free group $F_k=F(\Sigma_k)$ into the ring of formal power series with integral coefficients in noncommuting variables from $\Sigma_k=\{s_1,\dots,s_k\}$. The monomials over $\Sigma_k$ are ordered by short-lex and an element $u$ from $F_k$ is declared positive if and only if the coefficient in front of the smallest monomial (different from 1) in the power series representing $u$ under the Magnus embedding is positive.

For every finitely generated free group we construct an explicit left order extending the lexicographic order on the free monoid generated by the positive letters. The order is defined by a left, free action on the orbit of 0 of a free group of piecewise linear homeomorphisms of the line. The membership in the positive cone is decidable in linear time in the length of the input word by straightforward counting of subwords of length 2 of certain type (see the positivity criterion in Theorem~\ref{t:order-k}). The positive cone forms a context-free language closed under word reversal.

\subsection*{A free subgroup of $\PLFp(S^1)$ and its lift to $\PLFp(\RR)$}

Subdivide the circle $S^1=\RR/\ZZ=[0,1]/\{0=1\}$ into 7 arcs of the same length, denoted $1'=[0,1/7)$, $a'=[1/7,2/7)$, $b'=[2/7,3/7)$, $c'=[3/7,4/7)$, $C'=[4/7,5/7)$, $B'=[5/7,6/7)$, $A'=[6/7,1)$, and define three piecewise linear, orientation preserving homeomorphisms $a$, $b$, and $c$ of the circle (with finitely many breaks) as in the left half of Figure~\ref{f:f3circle}.
\begin{figure}[!ht]
\begin{tabular}{cc}
\includegraphics[width=225pt]{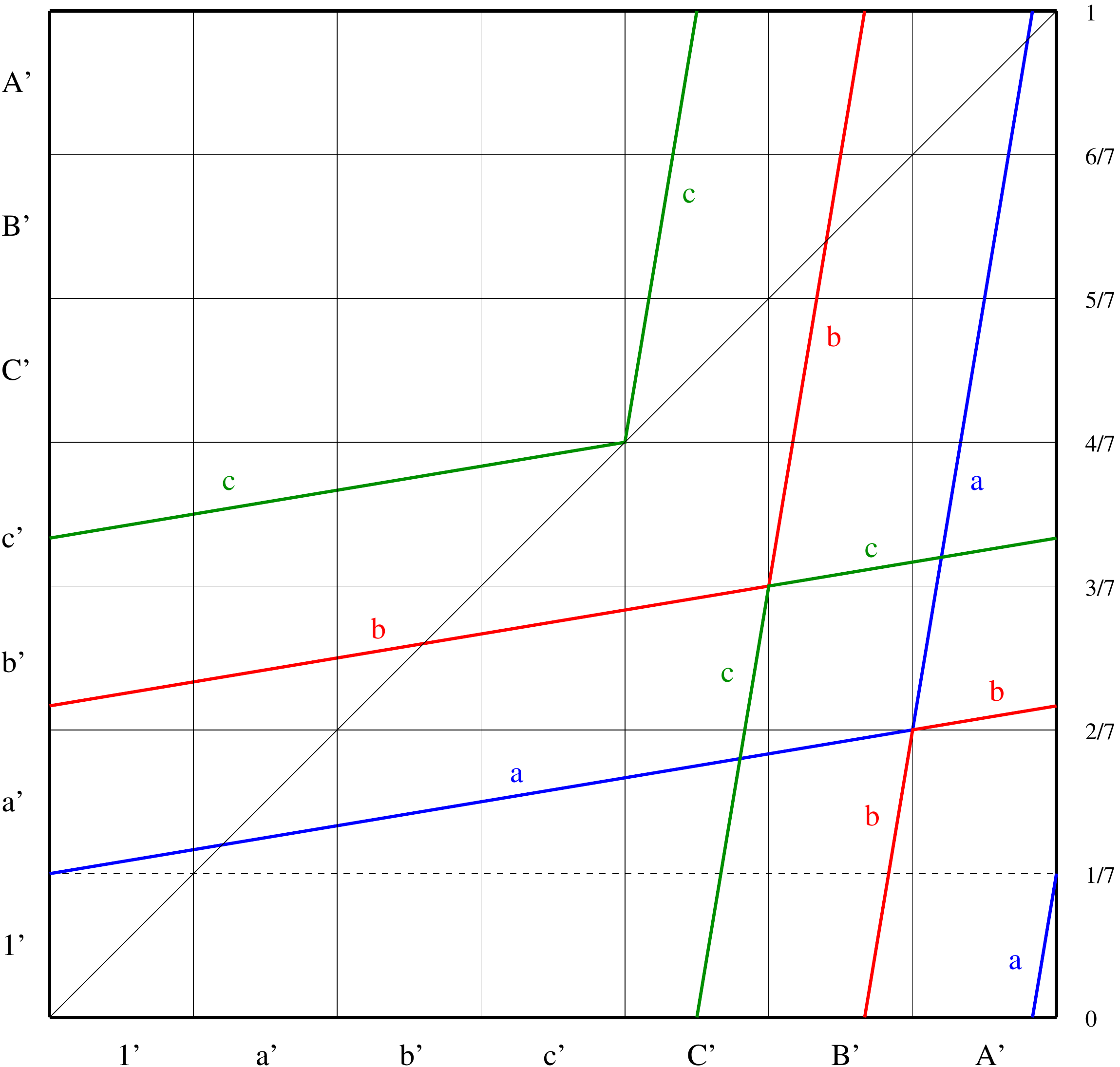} & \quad
\includegraphics[width=225pt]{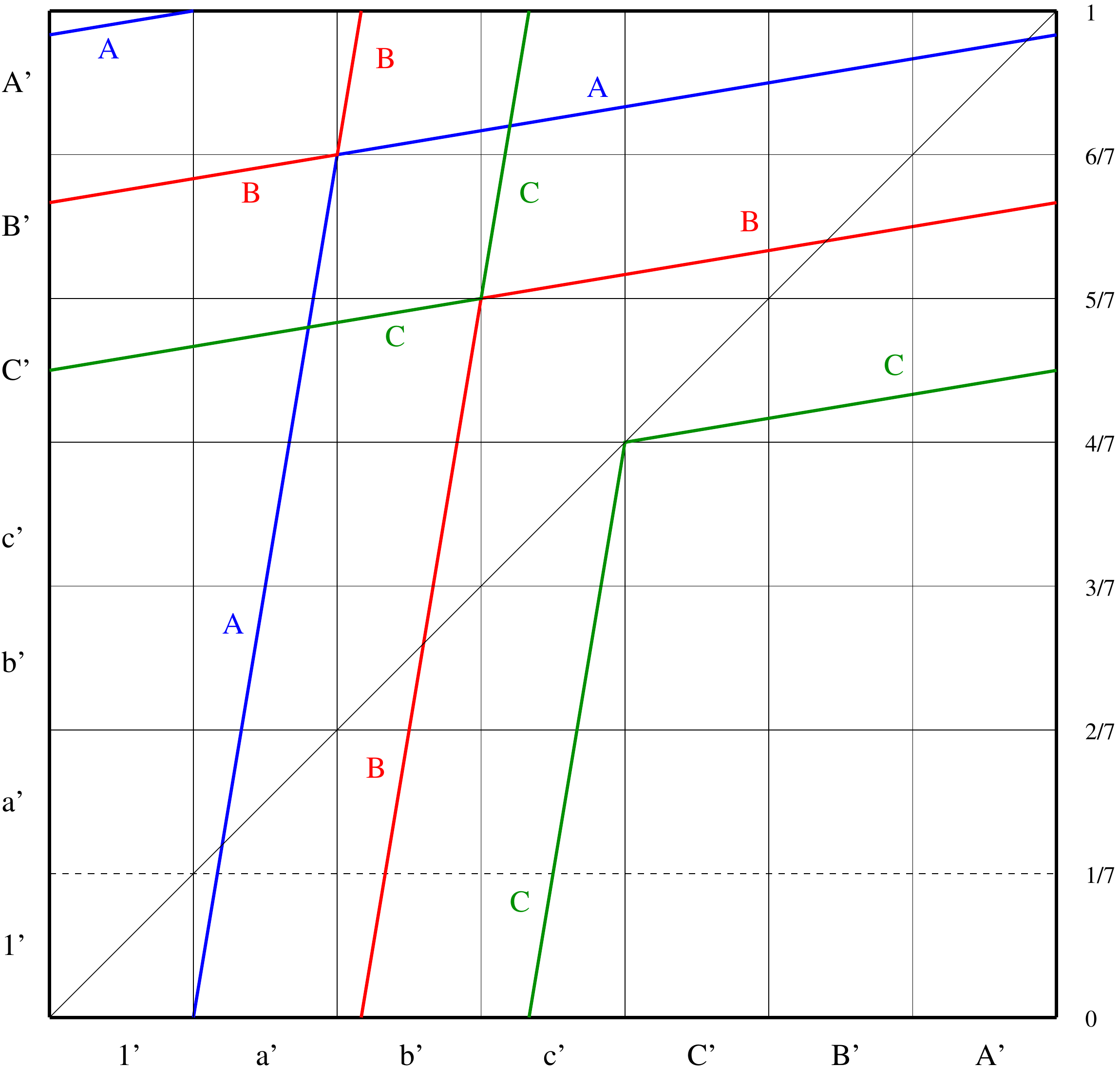}
\end{tabular}
\caption{Generators of $F$ in $\PLFp(S^1)$}
\label{f:f3circle}
\end{figure}
The inverses of $a$, $b$ and $c$ are denoted by $A$, $B$, and $C$, respectively, and are shown in the right half of Figure~\ref{f:f3circle}. It is evident that
\begin{alignat*}{6}
  &a(S^1 \setminus A') &&\subseteq a' \qquad \qquad
 &&b(S^1 \setminus B') &&\subseteq b' \qquad \qquad
 &&c(S^1 \setminus C') &&\subseteq c' \\
  &A(S^1 \setminus a') &&\subseteq A' \qquad \qquad
 &&B(S^1 \setminus b') &&\subseteq B' \qquad \qquad
 &&C(S^1 \setminus c') &&\subseteq C'
\end{alignat*}
Therefore, $F=\langle a,b,c\rangle$ is free of rank 3. Moreover, $F$ acts freely, through a left action, on the subset $F1'$ of the circle (note that $F1'$ has Lebesgue measure 1). In particular, it acts freely on the orbit of 0.

We lift the subintervals from the circle $S^1$ to subsets of the line $\RR$ along the projection $\RR \to \RR/\ZZ = S^1$, and we lift the maps $a$, $b$ and $c$ to piecewise linear, orientation preserving homeomorphisms of the line (with finitely many breaks on every compact subset) as in Figure~\ref{f:f3line}.
\begin{figure}[!ht]
\begin{tabular}{cc}
\includegraphics[width=225pt]{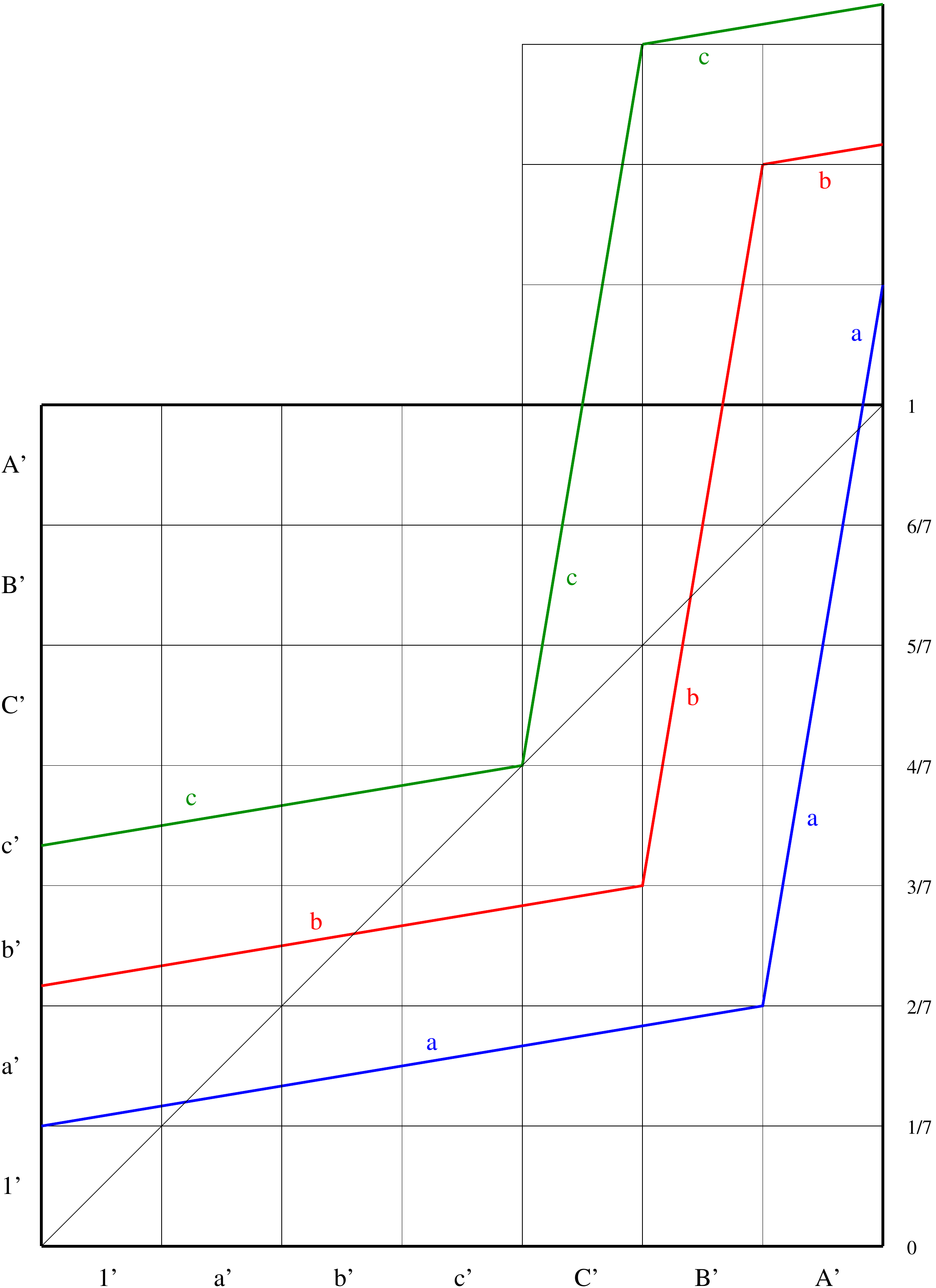} & \quad
\includegraphics[width=225pt]{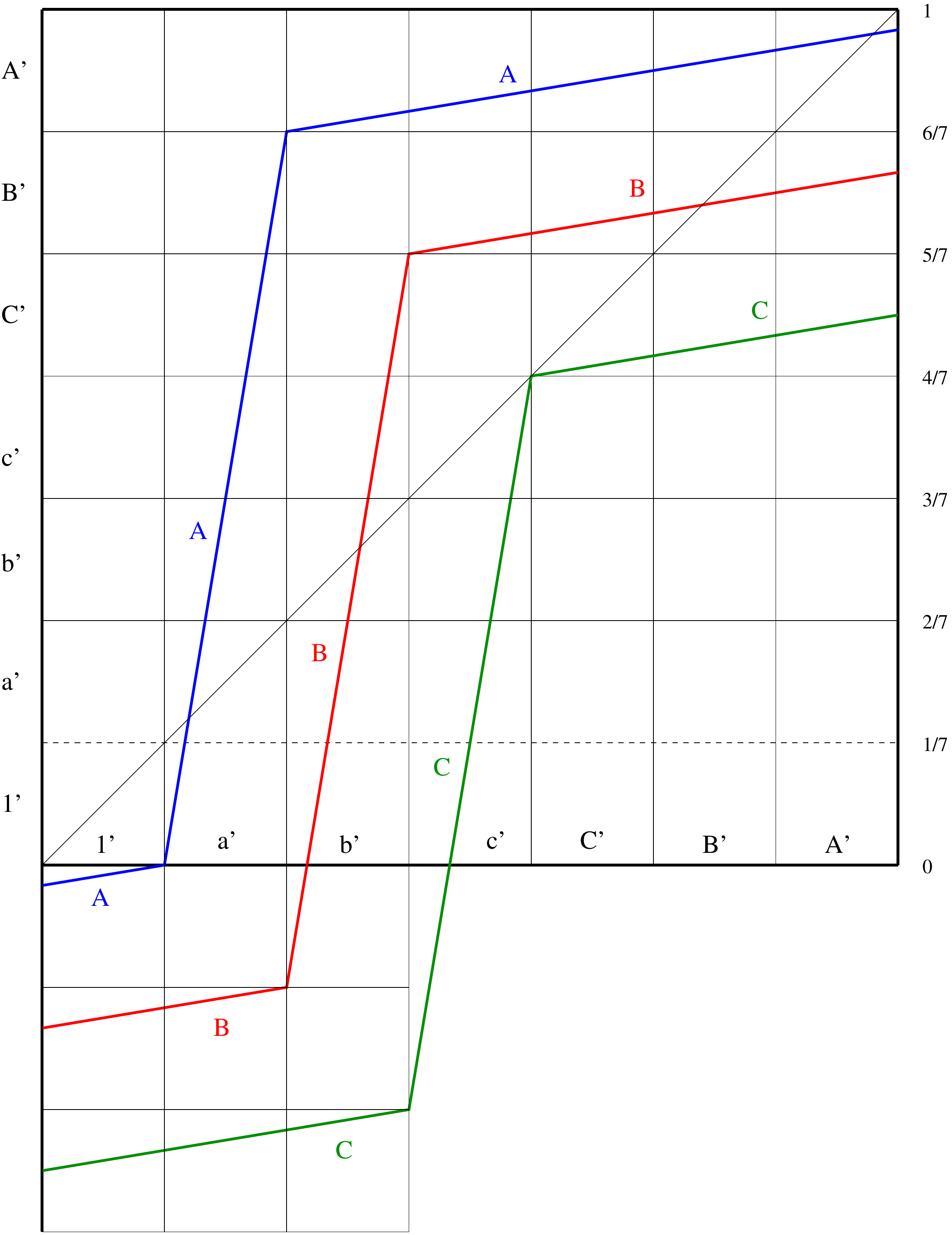}
\end{tabular}
\caption{Generators of $F$ in $\PLFp(\RR)$}
\label{f:f3line}
\end{figure}
We do not change the notation for the lifted subsets or homeomorphisms.


\subsection*{A left order on $F_3$}

Note that the lifted group $F$ is still free, acting freely on the orbit of 0. Therefore, a left order is defined on $F$ by declaring $g > h$ if and only if $g(0) > h(0)$ (equivalently, if and only if $h^{-1}g(0)>0$). A simple criterion for positivity is obtained by ``tracing'' the action on 0, i.e., by calculating, for reduced group words $u$, the signed distance, within error smaller than 1/2, between $u(0)$ and 0. Of course, since the action is given explicitly and the orbits of rational points are rational, $u(0)$ can be calculated exactly, but that takes longer, obscures the features of the order, and is not necessary.

\begin{proposition}\label{t:main}
Define a weight function $w:F \to \RR$ by
\begin{alignat*}{3}
 w(u) &= &&  &&\#(\textup{of subwords of }u\textup{ of the form } cB,~cA,\textup{ or }bA) - \\
      &  &&- &&\#(\textup{of subwords of }u\textup{ of the form } Cb,~Ca,\textup{ or }Ba) + \\
      &  &&+ &&\frac{1}{2}
        \begin{cases}
          1, &\textup{if } u \textup{ ends in a positive letter, i.e., one of the letters } a,b,c, \\
         -1, &\textup{if } u \textup{ ends in a negative letter, i.e., one of the letters } A,B,C, \\
          0, &\textup{if } u \textup{ is the trivial word}
        \end{cases},
\end{alignat*}
where $u$ is a reduced group word over $\{a,b,c\}$. Then
\[
 u(0)> 0 \qquad \text{ if and only if } \qquad w(u)>0.
\]
\end{proposition}

\begin{proof}
For any nontrivial word $u$, the point $u(0)$ is in exactly one of the intervals
\[ \dots,(-2,-1),(-1,0),(0,1),(1,2),\dots~. \]
We claim that $w(u)$ represents the midpoint of the interval to which $u(0)$ belongs. By definition, the action of $F$ is a left action, hence the last letter of $u$ acts first on 0 and pushes it into the interval (0,1) if the letter is positive, or to (-1,0) if the letter is negative, and the part of the formula for $w$ related to the last letter of $u$ records this as 1/2 or -1/2. From here on we trace the action of $u$ (from right to left) on the obtained point, but only record jumps into the next interval up (by adding 1 in our weight) and jumps into the next interval down (by subtracting 1 in our weight). Negative letters never produce jumps up (see the right half of Figure~\ref{f:f3line}) and positive letters never produce jumps down (see the left half of Figure~\ref{f:f3line}). By examining the action of the positive letters in the left half of Figure~\ref{f:f3line} we see that jumps up occur exactly when the letter $c$ is applied to a point in the regions $B'$ and $A'$, which is precisely when a subword of the form $cB$ or $cA$ occurs in $u$, or when the letter $b$ is applied to a point in the region $A'$, which is precisely when a subword of the form $bA$ occurs in $u$. We exclude the possibility of applying the letter $c$ to the region $C'$, $b$ to the region $B'$, and $a$ to the region $A'$, because $u$ is a reduced word. Similarly, by examining the action of the negative letters in the right half of Figure~\ref{f:f3line} we see that jumps down occur exactly when the letter $C$ is applied to a point in the regions $b'$ and $a'$, which is precisely when a subword of the form $Cb$ or $Ca$ occurs in $u$, or when the letter $B$ is applied to a point in the region $a'$, which is precisely when a subword of the form $Ba$ occurs in $u$. We exclude the possibility of applying the letter $C$ to the region $c'$, $B$ to the region $b'$, and $A$ to the region $a'$, because $u$ is a reduced word. We also exclude the possibility of applying any of the negative letters to the region $1'$, since the point 0 will never return to the region $1'$ under the action of a nontrivial word $u$, after, in the very first step, the last letter of $u$ moves it from there.
\end{proof}

We claim that the given order extends the usual lexicographic order on the free monoid $M_3=\{a,b,c\}^*$ based on $a<b<c$. All we need to verify is that, for all words $v_1,v_2,v_3$ in $M_3$, $e<av_1<bv_2<cv_3$, i.e., we need to verify that $w(av_1),w(v_1^{-1}Abv_2),w(v_2^{-1}Bcv_3)>0$. Since the weight of each of these words is 1/2, the claim is correct.

In fact, is is possible to see that the order, restricted to words in $M_3=\{a,b,c\}^*$ is lexicographic just by looking at the left half of Figure~\ref{f:f3line}. Namely, for $u$ in $M_3$, $u(0)$ is trapped in the interval $[0,4/7)$, and, for all words $v_1,v_2,v_3$ in $M_3$, we have $0<av_1(0)<bv_2(0)<cv_3(0)$, since, on the interval $[0,4/7)$ the entire graph of the function $a$ is above $0$ and below the minimum of the function $b$, and the entire graph of the function $b$ is below the minimum of the function $c$.


\subsection*{A left order on $F_k$}

Following an analogous construction (subdividing $S^1$ into $2k+1$ pieces, for $k \geq 2$, defining $k$ piecewise linear homeomorphisms $s_1,\dots,s_k$ of $S^1$ using only slopes $1/2k$ and $2k$, and so on) we may easily establish the following result.

\begin{theorem}\label{t:order-k}
Let $\Sigma_k=\{s_1,\dots,s_k\}$, for some $k \geq 2$. For $i=1,\dots,k$, denote $S_i=s_i^{-1}$. A left order on the free group $F_k =F(\Sigma_k)$ extending the lexicographic order on the free monoid $\Sigma_k^*$ based on the order $s_1<s_2<\dots<s_k$ on the alphabet of positive letters may be defined as follows. Define a weight map $w: F_k \to \RR$ by
\begin{alignat*}{3}
 w(u)
  &= && &&\#\{\textup{of subwords of }u \textup{ of the form }s_jS_i, \textup{ for } j>i\} - \\
  &  &&-&&\#\{\textup{of subwords of }u \textup{ of the form }S_js_i, \textup{ for } j>i\} + \\
  &  &&+&&\frac{1}{2}
    \begin{cases}
      1, &\textup{if }u \textup{ ends in any positive letter } s_i,~i=1,\dots,k, \\
     -1, &\textup{if }u \textup{ ends in any negative letter } S_i,~i=1,\dots,k, \\
      0, &\textup{if }u \textup{ is the trivial word}
    \end{cases},
\end{alignat*}
for a reduced word $u$ over $\Sigma_k$, and d eclare that the set
\[
 P_k = \{u \in F_k \mid w(u)>0 \}
\]
is the positive cone of $F_k$ (i.e., $u > v  \text{ in } F_k \Leftrightarrow w(v^{-1}u) > 0$).
\end{theorem}

The $k$ homeomorphisms $s_1,\dots,s_k$ of $S^1$ needed in the construction may be defined as follows. Define the homeomorphism $s_0$ of $S^1$ by
\[
 s_0(x) =
  \begin{cases}
    \frac{1}{2k} x, & 0 \leq x < \frac{2k}{2k+1} \\
    2k x - (2k-1), & \frac{2k}{2k+1} \leq x < 1,
  \end{cases}
\]
and, for $i=1,\dots,k$, define
\[
 s_i = s_0 \left(x+\frac{i-1}{2k+1}\right) + \frac{i}{2k+1}.
\]

Note that the membership problem in the positive cone $P_k$ is rather easy and can be solved in linear time in the length of the input word. If we count the relevant subwords as we read, we can calculate the weight and tell if a word is in the positive cone by the time we finish reading the word.

It is known that the positive cone of a left order of a free group cannot be finitely generated as a monoid (this can be deduced from the work of McCleary~\cite{mccleary:not-fg}, but was the first explicit proof is given by Navas~\cite{navas:ordered-dynamics}), which means that it cannot be a regular language of the form $Y^*$ for some finite set $Y$ of group words. On the other hand, it is apparent that the positive cone $P_k$ is a context free language over $\Sigma_k^\pm$ (indeed, a push down automaton with a single stack can easily establish if the number of subwords of the form $s_jS_i$, with $j>i$ in a given word $u$ is greater than, smaller than, or equal to the number of subwords of the form $S_js_i$, with $j>i$, and can take into account the last letter of $u$ in case of a tie).

Another interesting feature of the positive cone $P_k$ is that it is closed under word reversal. Indeed, for a given word $u$, with $w(u)>0$, we have $w\left((u^R)^{-1}\right) = -w(u)<0$, where $u^R$ denotes the word reversal of $u$, since the transformation $(u^R)^{-1}$ just exchanges the positive and negative letters (while keeping them in the same order as in $u$), and the effect of this on the weight is to exactly exchange all positive and all negative contributions. Since $w\left((u^R)^{-1}\right)<0$, we must have $w(u^R)>0$.

The left order provided in Theorem~\ref{t:order-k} is not two-sided (unlike the Magnus-Bergman order, which is). In fact, it is clear that no order extending the lexicographic order can be two-sided (this is because $s_1<s_1s_1$, but $s_1s_2 > s_1s_1s_2$).

Variations of the construction presented above lead to other explicitly stated orders on free groups, not necessarily extending the lexicographic order (for instance, in case of $F_2$, add 1 to the weight for every subword of the form $ab$ or $aB$, subtract 1 for every subword of the form $BA$ or $Ba$, and add plus or minus 1/2 depending on the last letter). 

\def\cprime{$'$}


\end{document}